\documentclass{article}
\usepackage{lmodern}
\usepackage{amsmath,amsfonts,amssymb,mathrsfs,amsthm}

 \title{Orthogonality of matrices in the Ky Fan $k$-norms\\ \small{Priyanka Grover \thanks{Email:  priyanka.grover@snu.edu.in}\thanks{ The author is supported by the research grant of INSPIRE Faculty Award of Department of Science and Technology, India.}\\ {\em Department of Mathematics, School of Natural Sciences, Shiv Nadar University, Gautam Buddha Nagar, Uttar Pradesh 201314, India}}}
\date{}
\newcommand{\Wk}{\mathcal W}
\newcommand{\W}{\mathscr W}
\newcommand{\X}{\mathscr X}
\newcommand{\Y}{\mathscr Y}
\newcommand{\mat}{\mathbb{M}(n)}

\newcommand{\C}{\mathbb{C}}
\newcommand{\R}{\mathbb{R}}
\newcommand{\tr}{\mathop{{\rm tr}}}
\newcommand{\conv}{\mathop{{\rm conv}}}
\newcommand{\h}{\mathop{{\rm Re}}}

\newtheorem{theorem}{Theorem}[section]

\theoremstyle{definition}
\newtheorem{defn}[theorem]{Definition}
\newtheorem{corollary}[theorem]{Corollary}

\newtheorem{prop}[theorem]{Proposition}

\newcommand{\Hil}{\mathcal{H}}

\newtheorem{lemma}[theorem]{Lemma}

\begin{document}
\maketitle
\begin{abstract}
We obtain necessary and sufficient conditions for a matrix $A$ to be Birkhoff-James orthogonal to another matrix $B$ in the Ky Fan $k$-norms. A characterization for $A$ to be Birkhoff-James orthogonal to any subspace $\W$ of $\mat$ is also obtained.
\end{abstract}



\numberwithin{theorem}{section}
\numberwithin{equation}{section}

\section{Introduction}
Let $\mat$ be the space of $n\times n$ complex matrices. Let $\|\cdot\|$ be any norm on $\mat$. Let $A, B\in \mat$. Then $A$ is said to be (Birkhoff-James) orthogonal to $B$ in $\|\cdot\|$ if \begin{equation}\|A+\lambda B\|\geq \|A\| \text{ for all }\lambda\in \C.\label{0.8}\end{equation}
In \cite{bhatiasemrl}, Bhatia and \v{S}emrl obtained a characterization for $A$ to be orthogonal to $B$ in the operator norm (also known as the spectral norm) $\|\cdot\|_{\infty}$. They showed that $A$ is orthogonal to $B$ in $\|\cdot\|_{\infty}$ if and only if there exists a unit vector $x \in \C^n$ such that $\|Ax\|=\|A\|_{\infty}$ and $\langle Ax, Bx \rangle=0.$  (All inner
products in this note are conjugate linear in the first component and linear in the second component.) Different proofs for this result have been studied in  \cite{tbgrover,keckic,lischneider}. This result can be restated as follows. If $A=U|A|$ is a polar decomposition of $A$, then $A$ is orthogonal to $B$ in $\|\cdot\|_{\infty}$ if and only if there exists a unit vector $x \in \C^n$ such that $|A|x=\|A\|_{\infty}x$ and $\langle x, U^*Bx \rangle=0.$ 
In \cite{bhatiasemrl}, it was also showed that if $\tr U^* B=0$ then $A$ is orthogonal to $B$ in the trace norm $\|\cdot\|_{1}$. And the converse is true if $A$ is taken to be invertible. 
Later, Li and Schneider \cite{lischneider} gave a characterization for orthogonality in $\|\cdot\|_{1}$ when $A$ need not be necessarily invertible. They showed the following. Let the number of  zero singular values of $A$ be $\ell$. Let $A=US V^*$ be a singular value decomposition of $A$. Let $B=U\left[\begin{array}{ccc} B_{11} &B_{12} \\ B_{21} & B_{22}  \end{array}\right]V^*,  \text{ where } B_{11}\in \mathbb M(n-\ell),  B_{22}\in \mathbb M(\ell).  $  Then
$\|A+\lambda B\|_{1}\geq \|A\|_{1}\text{ for all }\lambda \in \C$ if and only if $|\tr B_{11}|\leq \| B_{22}\|_1.$ 
% The proof shows that the latter condition is equivalent to saying that there exists $T\in \mathbb M(\ell)$ such that $\|T\|_{\infty}\leq 1$ and $\tr B_{11}+\tr (T^* B_{22})=0$.

The trace norm and the operator norm are special cases of two classes of norms, namely the Schatten $p$-norms $\|\cdot\|_p$ and the Ky Fan $k$-norms $\|\cdot\|_{(k)}$. In \cite{bhatiasemrl} and \cite{lischneider}, the authors have investigated the problem of finding necessary and sufficient conditions for orthogonality of matrices in $\|\cdot\|_p$, $1\leq p\leq \infty$. In this note, we obtain characterizations for orthogonality of matrices in $\|\cdot\|_{(k)}$, $1\leq k \leq n$.  Let $s_1(A)\geq s_2(A)\geq \cdots \geq s_n(A)\geq 0$ be the singular values of $A$. Then $\|A\|_{(k)}$ is defined as\begin{equation}\|A\|_{(k)}=s_1(A)+s_2(A)+\cdots+s_k(A).\end{equation} The cases $k=1$ and $k=n$ correspond to the operator norm $\|\cdot\|_{\infty}$ and the trace norm $\|\cdot\|_{1}$, respectively. We show the following.

\begin{theorem}\label{thm 1}
Let $A=U |A|$ be a polar decomposition of $A$.  If
there exist $k$ orthonormal vectors $u_1,u_2,\ldots,u_k$ such that 
\begin{equation}
|A|\ u_i=s_i(A)u_i \text{ for all } 1\leq i\leq k\label{cond1}
\end{equation}
and 
\begin{equation}
\sum_{i=1}^k \langle u_i, U^* B u_i\rangle=0, \label{cond2}
\end{equation} then
$A$ is orthogonal to $B$ in $\|\cdot\|_{(k)}$.
If $s_k(A)>0$, then the converse is also true.
\end{theorem}

The next theorem gives a more general characterization. 

\begin{theorem}\label{thm 2}
  Let $A=US V^*$ be a singular value decomposition of $A$.  Let the multiplicity of $s_k(A)$ be $r+q$, where $r\geq 0$ and $q\geq 1$, such that
$$s_{k-q+1}(A)=\cdots=s_{k+r}(A).$$ Let $B=U\left[\begin{array}{ccc} B_{11} &B_{12} & B_{13}\\ B_{21} & B_{22} & B_{23} \\ B_{31} & B_{32} & B_{33} \end{array}\right]V^*, \text{ where } B_{11}\in \mathbb M(k-q),  B_{22}\in \mathbb M(r+q),  B_{33}\in \mathbb M(n-k-r).$
\begin{enumerate}
\item[(a)] Let $s_k(A)>0$. Then $A$ is orthogonal to $B$ in $\|\cdot\|_{(k)}$ if and only if there exists a positive semidefinite matrix $T\in \mathbb M(r+q)$ with $\lambda_1(T)\leq 1$ and $ \sum_{j=1}^{r+q} \lambda_j(T)=q $  such that 
  $\tr B_{11}+\tr (T^* B_{22})=0.$
\item[(b)] Let $s_k(A)=0$. Then $A$ is orthogonal to $B$ in $\|\cdot\|_{(k)}$ if and only if there exists $T\in \mathbb M(n-k+q,r+q)$ with $s_1(T)\leq 1$, and $\sum_{j=1}^{r+q} s_j(T)\leq q$
such that  $\tr B_{11}+\tr \left(T^* \left[\begin{array}{ccc}B_{22}\\ B_{32}\end{array}\right]\right)=0.$
\end{enumerate}
\end{theorem}

Let $\W$ be any subspace of $\mat$. Then $A$ is said to be orthogonal to $\W$ (in the Birkhoff-James sense) in a given norm $\|\cdot\|$ on $\mat$ if \begin{equation}\|A+W\|\geq \|A\|  \text{ for all } W \in \W. \label{orthogonalitytoW}\end{equation}
In \cite{grover}, we obtained a necessary and sufficient condition for $A$ to be orthogonal to $\W$ in the operator norm. Our next theorem gives a characterization for $A$ to be  orthogonal to $\W$ in $\|\cdot\|_{(k)}$.
\begin{theorem}\label{thm 3}
Let $A=U |A|$ be a polar decomposition of $A$. Let $\W$ be any subspace of $\mat.$ If there exist density matrices $P_1,P_2,\ldots,P_k$ such that $\|\sum_{i=1}^k P_i\|_{\infty}\leq 1$, $|A|P_i=s_i(A) P_i \ (1\leq i\leq k)$ and $U \sum_{i=1}^k P_i \in \W^{\perp}$, then $A$ is orthogonal to $\W$ in $\|\cdot\|_{(k)}$. If $s_k(A)>0$, then the converse is also true. \end{theorem}

If $m_i(A)$ is the multiplicity of $s_i(A)$, then the condition $|A|P_i=s_i(A) P_i$ implies that the range of $P_i$ is a subspace of the eigenspace of $|A|$ corresponding to $s_i(A)$. So $\text{rank } P_i$ is at most $m_i(A)$. 

The problem of finding characterizations of orthogonality of a matrix to a subspace $\W$  of $\mat$ is closely related to the best  approximation problems \cite{singer}. A specific question is when is the zero matrix a best approximation to $A$ from $\W?$ This is the same as asking when is $A$ orthogonal to $\W$? 

In \cite{lischneider}, the authors studied a characterization for orthogonality in the induced matrix norms.  Ben\'itez, Fern\'andez and Soriano \cite{benitez} showed that a necessary and sufficient condition for the norm of a real finite dimensional normed space $\mathscr X$ to be induced by an inner product is that for any bounded linear operators $A,B$ from $\mathscr X$ into itself, $A$ is orthogonal to $B$ if and only if there exists a unit vector $x\in X$ such that $\|Ax\|=\|A\|$ and $\langle Ax, Bx\rangle=0$. More results in this direction have been obtained recently in \cite{sainpaul, sainpaulhait}. Characterizations of orthogonality on Hilbert C$^*$-modules have been studied in \cite{rajic1, rajic2, rajic3, tbgrover}.

To obtain the proofs of the above theorems, we use methods that we had introduced  in \cite{tbgrover} and \cite{grover}. 
%I shall be reaping fruits of the earlier labour done there. 
We first obtain some new expressions for the subdifferential of the map taking a matrix $A$ to its Ky Fan $k$-norm $\|A\|_{(k)}$  in Section \ref{preliminaries}.  The proofs of the above theorems are given in Section \ref{proofs} followed by some remarks in Section \ref{remarks}.

\section{Subdifferentials of the Ky Fan $k$-norm}\label{preliminaries} 
Let $\X$ be a Banach space and let  $f:\X\rightarrow \R$ be a convex function.

\begin{defn}
 A subgradient of $f$ at $a \in \X$ is an element $\varphi$ of the dual space $ \X^*$ such that 
\begin{equation}f(y)-f(a)\geq \h \varphi(y-a) \quad \text{for all } y\in \X.\label{eq 1.3.1}\end{equation}
\end{defn}

The \emph{subdifferential} of $f$ at $a$  is the set of bounded linear functionals $\varphi \in \X^*$ satisfying \eqref{eq 1.3.1} and is denoted by $\partial f(a)$. It is a non-empty weak* compact convex subset of $\X^*$. For more details, see \cite[Chapter D]{hiriart} and \cite[Chapter 2]{zalinescu}.
The following proposition is a direct consequence of the definition of the subdifferential. It is one of the most useful tools that we require in Section \ref{proofs}.

\begin{prop}\label{p2.3}
A continuous convex function $f:\X\rightarrow \R$ attains its minimum value at $a$ if and only if $0\in \partial f(a)$.
\end{prop}
An equivalent definition of the subdifferential of a continuous convex function can be given in terms of $f'_+(a,x)$, the right directional derivative of $f$ at $a$ in the direction $x$:
%\begin{equation}f'_+(a,x)=\lim_{t\downarrow 0}\frac{f(a+tx)-f(a)}{t}.\label{eq 1.2.17}\end{equation}
\begin{equation}
\partial f(a)=\{\varphi\in \X^*: \h\ \varphi(x)\leq f'_+(a,x) \text{ for all } x\in \X\}.\end{equation}
Moreover, for each $x\in \X$, 
\begin{equation}
f'_+(a,x)=\max\{\h \varphi(x):\varphi\in \partial f(a)\}. \label{defn2}
\end{equation}

The following rule of subdifferential calculus will be helpful in our analysis later. 
\begin{prop}\label{subdiff chain rule}
Let $\X$ and $\Y$ be Banach spaces. Let $S:\X\rightarrow \Y$ be a bounded linear map and let $L:\X\rightarrow \Y$ be the continuous affine map defined by $L(x)=S(x)+y_0$, for some $y_0\in \Y$. Let $g:\Y\rightarrow \R$ be a continuous convex function. Then
\begin{equation}
\partial(g\circ L)(a)=S^* \partial g(L(a)) \text{ for all } a\in \X,\label{eq 1.3.9}
\end{equation}
where $S^*$ denotes the real or complex adjoint of $S$ (depending on whether $\X$ and $\Y$ are both real or both complex Banach spaces.)
\end{prop}

 For any norm $\|\cdot\|$  on the space $\mat$, it is well known that
\begin{equation}
\partial \|A\|=\{G\in \mat: \|A\|=\h \tr(G^* A),\|G\|^*\leq 1\},\label{eq 1.4.1}
\end{equation}
where $\|\cdot\|^*$ is the dual norm of $\|\cdot\|$, and \begin{equation}
\|T\|=\sup_{\|X\|^*=1} |\tr(T^* X)|=\sup_{\|X\|^*=1} \h \tr(T^* X).\label{supnorm}
\end{equation} 
 The subdifferentials of  some classes of matrix norms, namely unitarily invariant norms and induced norms, have been computed by Watson \cite{watson}. 
The following expression for the subdifferential of the Ky Fan $k$-norms was also given by him in \cite{watson93}. 
Let $1\leq k\leq n$. Let the multiplicity of $s_k(A)$ be $r+q$, where $r\geq 0$ and $q\geq 1$, such that
$$s_{k-q+1}(A)=\cdots=s_{k+r}(A).$$ 
Let $g:\mat\rightarrow \R$ be the function defined as $g(A)=\|A\|_{(k)}$.
\begin{theorem}[{{\cite{watson93}}}]
\label{watson ky fan}
 Let $A=USV^*$ be a singular value decomposition of $A$ and let the matrices $U, V$ be partitioned as $U=[U_1: U_2:U_3]$ and $V=[V_1:V_2:V_3]$ where $U_1,V_1\in \mathbb M(n, k-q); U_2,V_2\in \mathbb M(n, r+q); U_3,V_3\in \mathbb M(n, n-k-r)$. If $ s_k(A)>0 $, then $G\in \partial g(A)$ if and only if there exists a positive semidefinite matrix $T\in \mathbb M(r+q) $ with $\lambda_1(T)\leq 1 \text{ and }  \sum_{j=1}^{r+q} \lambda_j(T)=q$ such that 
$G=U_1V_1^*+U_2TV_2^*.$ If $ s_k(A)=0$, then $G\in \partial g(A)$ if and only if there exists $T\in \mathbb M(n-k+q, r+q) $ with $s_1(T)\leq 1 \text{ and }  \sum_{j=1}^{r+q} s_j(T)\leq q$ such that 
$G=U_1V_1^*+[U_2:U_3]TV_2^*.$
%\begin{enumerate} \item[(a)] Let $s_k(A)\neq 0$. Then $G\in \partial \|A\|_k$ if and only if there exists $T\in \mathbb M(s+t)$ such that $$G=U_1V_1^*+U_2TV_2^*,$$ where $T$ is positive semidefinite, the singular values of $T$ are at most 1 and the sum of singular values of $T$ is equal to $t$. \item[(b)] Let $s_k(A)=0$.  Then $G\in \partial \|A\|_k$ if and only if there exists $T\in \mathbb M(n-k+t,s+t)$ such that $$G=U_1V_1^*+U_2TV_2^*,$$ where the singular values of $T$ are at most 1 and the sum of singular values of $T$ is at most $t$. \end{enumerate}
\end{theorem}
 We obtain new formulas for $\partial g(A)$ that can be used more easily in our problem. The computations are similar to the ones in \cite{watson}. To do so, we first calculate $g'_+(A,\cdot)$. For this, an important thing to observe is that the Ky Fan $k$-norm of a matrix $A$ is also given by 
\begin{eqnarray}
\|A\|_{(k)}&=&\max _{\substack{U,V\in \mathbb M(n,k)\\U^*U=V^*V=I_k}}\h\tr U^* A V=\max_{\substack{U,V\in \mathbb M(n,k)\\U^*U=V^*V=I_k}}|\tr U^* A V|. \label{fan}
\end{eqnarray}
(See \cite[p. 791]{marshallolkin}.)
If $A$ is positive semidefinite, then
\begin{equation}
\|A\|_{(k)}=\max _{\substack{U\in \mathbb M(n,k)\\U^*U=I_k}}\tr U^* A U.\label{fanherm}
\end{equation}

\begin{theorem}\label{thm dir der ky fan}
For  $X\in \mat$,
\begin{equation}
g'_+(A,X)
=\max_{\substack{ \\ u_1,\ldots,u_k\rm{\ o.n. }\\ v_1,\ldots,v_k\rm{\ o.n. }\\ Av_i=s_i(A) u_i}}\sum_{i=1}^k \h\langle u_i,Xv_i\rangle.\label{limit}
\end{equation}
\end{theorem}

\begin{proof}
From \eqref{fan}, we have
\begin{equation}
\|A\|_{(k)}=\max_{\substack{u_1,\ldots,u_k \text{ o.n. }\\ v_1,\ldots,v_k \text{ o.n. }}}\sum_{i=1}^k \h \langle u_i,A v_i\rangle.\label{fan2}
\end{equation}
For any sets of $k$ orthonormal vectors $u_1,\ldots,u_k$ and $ v_1,\ldots,v_k$ satisfying $ A v_i=s_i(A) u_i, 1\leq i\leq k$, we have
\begin{eqnarray*}
\|A+tX\|_{(k)}&\geq & \sum_{i=1}^k \h\langle u_i, (A+tX)v_i\rangle\\
&=& \sum_{i=1}^k s_i(A)+t \sum_{i=1}^k\h \langle u_i, Xv_i\rangle\\
&=& \|A\|_{(k)}+t \sum_{i=1}^k\h \langle u_i, X v_i\rangle.
\end{eqnarray*}
This gives for $t>0$,
\begin{equation}
 \frac{\|A+tX\|_{(k)}-\|A\|_{(k)}}{t}\geq \max_{\substack{ \\ u_1,\ldots,u_k\text{ o.n. }\\ v_1,\ldots,v_k\text{ o.n. }\\ Av_i=s_i(A) u_i}}\sum_{i=1}^k \h\langle u_i,Xv_i\rangle. \label{greaterthan}
\end{equation}
Now for any sets of $k$ orthonormal vectors $u_1(t),\ldots,u_k(t)$ and $ v_1(t),\ldots,v_k(t)$ satisfying 
\begin{equation}
(A+tX) v_i(t)=s_i(A+tX) u_i(t), \quad 1\leq i\leq k,\label{1'}
\end{equation} we have
\begin{eqnarray*}
\|A\|_{(k)}&\geq & \sum_{i=1}^k \h\langle u_i(t), Av_i(t)\rangle\\
&=& \sum_{i=1}^k s_i(A+tX) - t\sum_{i=1}^k\h \langle u_i(t), Xv_i(t)\rangle\nonumber\\
&=& \|A+tX\|_{(k)}-t \sum_{i=1}^k\h \langle u_i(t),X v_i(t)\rangle. 
 \end{eqnarray*}
So for each $t>0$, we obtain
\begin{equation}
 \frac{\|A+tX\|_{(k)}-\|A\|_{(k)}}{t}\leq \sum_{i=1}^k \h\langle u_i(t),Xv_i(t)\rangle. \label{lessthan}
\end{equation}
Consider a sequence $\{t_n\}$ of positive real numbers converging to zero as $n\rightarrow \infty$.  
Since the unit ball in $\C^n$  is compact, there exists a subsequence $\{t_{n_{m}}\}$ of $\{t_n\}$ such that for each $1\leq i\leq k$, there exist ${u'_i}$ and ${v'_i}$ such that $\{u_i(t_{n_m})\}$ and $\{v_i(t_{n_m})\}$ converge to $u'_i$ and $v'_i$, respectively, as $m\rightarrow \infty$. Then the sets of vectors $u'_1,\ldots,u'_k$ and $v'_1,\ldots,v'_k$ are orthonormal. By continuity of singular values, %using Weyl's perturbation theorem (see \cite[Corollary III.2.6]{bhatiama}), 
we also know that \begin{equation}
s_i(A+t_{n_m}B)\rightarrow s_i(A) \text{ as } m\rightarrow \infty.\label{sv}
\end{equation} Hence we obtain $A{v'_i}=s_i(A) {u'_i}$  for all $ 1\leq i\leq k$. By \eqref{lessthan}, we get that
\begin{equation}
g'_+(A,X)= \lim_{m\rightarrow \infty} \frac{\|A+t_{n_m}X\|_{(k)}-\|A\|_{(k)}}{t_{n_m}}\leq  \max_{\substack{ \\ u_1,\ldots,u_k\text{ o.n. }\\ v_1,\ldots,v_k\text{ o.n. }\\ Av_i=s_i(A) u_i}}\sum_{i=1}^k\h \langle u_i,Xv_i\rangle.
\end{equation}
Combining this with \eqref{greaterthan}, we obtain the required result. 
\end{proof}

The above proof works equally well if the maximum in \eqref{limit} is taken over the sets of orthonormal vectors $u_1,\ldots,u_k$ and $v_1,\ldots,v_k$  such that for each $1\leq i \leq k$, $u_i$ and $v_i$ are left and right singular vectors of $A$, respectively, corresponding to the $i$th singular value $s_i(A)$ of $A$. We note here that for each $t>0$, if along with \eqref{1'}, we also have $$(A+tX)^*u_i(t)=s_i(A+tX) v_i(t),$$ then by passing onto a subsequence $\{t_{n_m}\}$ as in the above proof, and taking the limit as $m\rightarrow \infty$,  we obtain $$A^* u'_i=s_i(A) v'_i.$$ So for each $X\in \mat$, we get
\begin{equation}
g'_+(A,X)=\max_{\substack{ \\ u_1,\ldots,u_k\text{ o.n. }\\ v_1,\ldots,v_k\text{ o.n. }\\ Av_i=s_i(A) u_i\\ A^* u_i=s_i(A) v_i}}\sum_{i=1}^k\h \langle u_i,Xv_i\rangle.\label{general}
\end{equation}

\begin{corollary}
Let $A$ be positive semidefinite. Let $\lambda_1(A)\geq \cdots \geq \lambda_n(A) \geq 0$ be the eigenvalues of $A$, with $\lambda_k(A)>0$.  Then
\begin{equation}
 g'_+(A,X)=\max_{\substack{ \\ u_1,\ldots,u_k\rm{\ o.n. }\\ Au_i=\lambda_i(A) u_i}}\sum_{i=1}^k\h \langle u_i,Xu_i\rangle. \label{subdiffps}
\end{equation}
\begin{proof}
We know that  if  $Av=\lambda u$ and $Au=\lambda v$, where $\lambda>0$,  then $u=v$. Using this, the required result follows from \eqref{general}. 
%\textcolor{blue}{This is because we have $A^2 v=\lambda^2 v$, which is the same as saying that $(A+\lambda I)(A-\lambda I) v=0.$ Since $A+\lambda I$ is invertible, this implies that $Av=\lambda v$. But we have that $Av=\lambda u$. If $\lambda \neq 0$, then it must be that $u=v$. }
\end{proof}
\end{corollary}

\begin{theorem}\label{thm subdiff ky fan}
Let $A\in \mat$. Then
\begin{eqnarray}
\partial g(A)%\|A\|_{(k)}
&=& \conv\Bigg\{\sum_{i=1}^k u_i v_i^*:u_1,\ldots,u_k,v_1,\ldots,v_k \in \C^n, u_1,\ldots,u_k {\rm{\ o.n.}}, v_1,\ldots,v_k  {\rm{\ o.n.}},\nonumber\\& & \hspace{3cm} A v_i=s_i(A) u_i \text{ for all }1\leq i\leq k\Bigg\}\label{subdifferential}\\
&=&  \conv\Bigg\{\sum_{i=1}^k u_i v_i^*:u_1,\ldots,u_k,v_1,\ldots,v_k \in \C^n, u_1,\ldots,u_k {\rm{\ o.n.}}, v_1,\ldots,v_k{\rm{\ o.n.}},\nonumber\\
& &  \hspace{2.7cm}  A v_i=s_i(A) u_i, A^* u_i=s_i(A) v_i \text{ for all }1\leq i\leq k\Bigg\}.\label{subdifferential2}
\end{eqnarray}
\end{theorem}
\begin{proof}
Denote the set on the right hand side of \eqref{subdifferential} by $\mathbb H(A)$.
%\begin{eqnarray*}\mathscr S(A)&=&\{\sum_{i=1}^k u_i v_i^*:u_1,\ldots,u_k,v_1,\ldots,v_k \in \C^n, u_1,\ldots,u_k \text{ o.n.},v_1,\ldots,v_k  \text{ o.n.}, A v_i=s_i(A) u_i\\& & \hspace{3cm}\text{ for all }1\leq i\leq k\}.\end{eqnarray*} 
 Let
$G\in \mathbb H(A)$. Then  $$G= \sum_{i=1}^k u_{i} v_{i}^*,$$ where $u_1,\ldots,u_k$ and $v_1,\ldots,v_k$ are orthonormal sets of vectors such that $Av_i=s_i(A)u_i$ for all $1\leq i\leq k$ . So
\begin{eqnarray*}
\h \tr (G^*A)%&=& \sum_{i=1}^k \h \tr (v_{i} u_{i}^*A)\\
&=& \sum_{i=1}^k \h \langle  u_{i}, A v_{i}\rangle\\
&=&  \sum_{i=1}^k s_i(A)\\
&=& \|A\|_{(k)},
\end{eqnarray*}

and\begin{eqnarray*}
\h \tr(G^* X)&=&\sum_{i=1}^k\h \langle  u_{i}, X v_{i}\rangle\\
&\leq& \|X\|_{(k)}.
\end{eqnarray*}
%The last inequality follows by \eqref{fan2}.
Thus
$$\|G\|^*\leq 1.$$
So we get by \eqref{eq 1.4.1} that $\mathbb H(A)\subseteq \partial g(A)$, and therefore $\conv \mathbb H(A)\subseteq \partial g(A)$.

 %$G\in \conv \mathscr S(A)$. Then  $$G=\sum_{p=1}^m t_p \sum_{i=1}^k u_{i}^{(p)} v_{i}^{{(p)}^*},$$ where $0\leq t_p\leq 1, \sum_{p=1}^m t_p=1$, $u_1^{(p)},\ldots,u_k^{(p)}$ and $v_1^{(p)},\ldots,v_k^{(p)}$ are orthonormal sets of vectors such that $Av_i^{(p)}=s_i(A)u_i^{(p)}$ for all $1\leq i\leq k$ . We have \begin{eqnarray*}\h \tr (G^*A)&=& \sum_{p=1}^m t_p\sum_{i=1}^k \h \tr (v_{i}^{(p)} u_{i}^{{(p)}^*}A)\\&=& \sum_{p=1}^m t_p \sum_{i=1}^k \h \langle  u_{i}^{(p)}, A v_{i}^{(p)}\rangle\\&=& \sum_{p=1}^m t_p \sum_{i=1}^k s_i(A)\\&=& \|A\|_{(k)}.\end{eqnarray*}And\begin{eqnarray*}\h \tr(G^* X)&=& \sum_{p=1}^m t_p \sum_{i=1}^k\h \langle  u_{i}^{(p)}, X v_{i}^{(p)}\rangle\\&\leq& \sum_{p=1}^m t_p\|X\|_{(k)} \\&=& \|X\|_{(k)}.\end{eqnarray*}This gives $$\|G\|^*\leq 1.$$

Now let $G\in \partial g(A)$. Suppose  $G\notin \conv \mathbb H(A)$. 
%\textcolor{blue}{Note that $\mathscr S(A)$ is the image of the set $$J'=\left\{(U,V)\in \mathbb M({n,k}): U^*U=V^*V=I_k, U^*AV=\Sigma_k=\left[\begin{array}{ccc}s_1&\ &\ \\\ &\ddots&\ \\\ &\ &s_k(A)\end{array}\right]\right\}$$ under the continuous map $(U,V)\rightarrow UV^*$.
%and the set $J'$  is compact.
%Let $J$ be the set as defined in \eqref{J}. Consider the map $F_1:(U,V)\rightarrow U^*AV-\Sigma_k$. It is continuous and $J'=F_1^{-1}\{{0_k}\}\cap J$.  Therefore $J'$ is a closed subset of $J$, and hence is compact.}
The set $\mathbb H(A)$ is compact, and  so is its convex hull. 
By  the Separating Hyperplane Theorem, there exists $X\in \mat$ such that  for all  sets of $k$ orthonormal vectors $u_1,\ldots,u_k$ and $v_1,\dots,v_k$ satisfying $Av_i=s_i(A) u_i$ for $ 1\leq i\leq k$, we have
$$\h \tr\left(X^* \left(\sum_{i=1}^k u_i v_i^*-G\right)\right)<0.$$
This implies
$$\max_{\substack{u_1,\ldots,u_k \text{ o.n.}\\ v_1,\ldots,v_k \text{ o.n.}\\ Av_i=s_i(A) u_i }} \sum_{i=1}^k \h \langle u_i, X v_i\rangle<\max_{G\in \partial g(A)}\h \tr(X^* G).$$
%\textcolor{blue}{Note that the supremum on the left hand side is attained, since the underlying set $J'$ is compact and $(U,V)\rightarrow \h \tr U^*XV$ is a continuous map.}
By \eqref{defn2}, the right hand side is %the support function of $\partial g(A)$, which is given by 
$g'_+(A,X)$. By \eqref{limit}, this should be equal to the left hand side. This gives a contradiction. Thus we obtain \eqref{subdifferential}.

The expression \eqref{subdifferential2} can be proved similarly by using \eqref{general}, instead of \eqref{limit}.

\end{proof}

\begin{corollary}\label{cor 1.5.15}
Let $A$ be a positive semidefinite matrix, with eigenvalues $\lambda_1(A)\geq \cdots\geq \lambda_n(A)\geq 0$ such that $\lambda_{k}(A)>0$. Then
\begin{equation}
\partial g(A)=\conv\left\{\sum_{i=1}^k u_i u_i^*:u_1,\ldots,u_k\in \C^n, u_1,\ldots,u_k \text{ o.n.},  Au_i=\lambda_i(A) u_i\text{ for all }1\leq i\leq k\right\}.\label{eq 1.5.29}
\end{equation}
\end{corollary}

\section{Proofs}\label{proofs}
To prove Theorem \ref{thm 1}, we require the following lemma.
\begin{lemma}\label{lemma 2.1.7}
Let $X,Y\in \mat$ and let $Y$ be positive semidefinite. Let $\lambda_1(Y)\geq \cdots \geq \lambda_n(Y)\geq 0$ be the eigenvalues of $Y$. For $1\leq r\leq n$, let $$\mathcal W(X,Y)={\left\{\sum_{i=1}^{r} \langle u_i, Xu_i\rangle: u_1,\ldots,u_r \in \C^n, u_1,\ldots,u_r \text{ o.n.},Y u_i=\lambda_i(Y) u_i   \text{ for all }1\leq i\leq r \right\}}.$$ 
Then $\mathcal W(X,Y)$ is a convex set. 
\end{lemma}
\begin{proof}
Let the number of distinct eigenvalues  of $Y$ 
 be $\ell$ and let $\Hil_1,\ldots,\Hil_\ell$ be the respective eigenspaces. Let  $m_1,\ldots,m_{\ell}$ be the dimensions of $\Hil_1,\ldots,\Hil_{\ell}$, respectively. Let $1\leq \ell' \leq \ell$ be such that $ m_1+\cdots+m_{\ell'-1}< r\leq m_1+\cdots+m_{\ell'}$. Let $m=r-(m_1+\cdots+m_{\ell'-1}).$ Set $$\Wk_j(X)=\left\{\sum_{i=1}^{m_j} \langle  u_i, X u_i\rangle: u_1,\ldots,u_{m_j} \in \Hil_j, u_1,\ldots,u_{m_j} \text{ o.n.}\right\} \text{ for } 1\leq j\leq \ell'-1,$$ and
$$\Wk_{\ell'}(X)=\left\{\sum_{i=1}^{m} \langle  u_i, X u_i\rangle: u_1,\ldots,u_{m} \in \Hil_{\ell'}, u_1,\ldots,u_{m} \text{ o.n.}\right\}.$$
Since $\Hil_1,\ldots,\Hil_\ell$ are mutually orthogonal, we have \begin{equation}
\Wk(X,Y)=\sum_{j=1}^{\ell'} \Wk_j(X).\label{3.1}
\end{equation}
 Note that $\Wk_j(X)$ is a singleton set for $1\leq j \leq \ell'-1$. Hence it is sufficient to show that $\Wk_{\ell'}(X)$ is convex. Let $\mathscr P_{\ell'}$ be the orthogonal projection from $\C^n$ onto $\Hil_{\ell'}$, and let $\iota_{\ell'}$ denote its adjoint (which is the inclusion map of $\Hil_{\ell'}$ into $\C^n$).  Then 
$\mathcal W_{\ell'}(X)$ is the $m$-numerical range of $\mathscr P_{\ell'} X\iota_{\ell'}$, which is convex (see \cite[p. 315]{halmos}).
 \end{proof}

We now state and prove a real version of Theorem \ref{thm 1}.

\begin{theorem}\label{thm 3.2}
Let $A=U |A|$ be a polar decomposition of $A$.  If
there exist $k$ orthonormal vectors $u_1,u_2,\ldots,u_k$ such that 
\begin{equation}
|A|\ u_i=s_i(A) u_i \text{ for all } 1\leq i\leq k\label{realcond1}
%\|A\|_{(k)}=\sum_{i=1}^k \langle u_i, P u_i\rangle \label{norm}
\end{equation}
and 
\begin{equation}
\sum_{i=1}^k \h\langle u_i, U^* B u_i\rangle=0, \label{realcond2}
\end{equation} then
\begin{equation}
\|A+t B\|_{(k)} \geq \|A\|_{(k)} \text{ for all } t\in \R. \label{real}
\end{equation}
If $s_k(A)>0$, then the converse is also true.
\end{theorem}

\begin{proof}
First suppose that there exist $k$ orthonormal vectors $u_1,u_2,\ldots,u_k$ such that 
$|A|\ u_i=s_i(A)\ u_i \text{ for all } 1\leq i\leq k$
and 
$\sum_{i=1}^k \h\langle u_i, U^* B u_i\rangle=0.$
We have $$\|A+t B\|_{(k)}= \||A|+t U^*B\|_{(k)}$$ and by \eqref{fan},
$$\||A|+t U^*B\|_{(k)}\geq \sum_{i=1}^k\h \langle u_i, (|A|+t U^*B) u_i\rangle.$$
So we get
\begin{eqnarray*}
\|A+t B\|_{(k)}&\geq& \sum_{i=1}^k \langle u_i, |A| u_i\rangle+ t \sum_{i=1}^k \h\langle u_i, U^*B u_i\rangle\\
&=& \sum_{i=1}^k s_i(A)\\
&=& \|A\|_{(k)}.
\end{eqnarray*}

Now suppose that $s_k(A)>0$ and $$\|A+tB\|_{(k)}\geq \|A\|_{(k)}\text{ for all }t\in \R.$$ This can also be written as  \begin{equation}\||A|+t U^* B\|_{(k)}\geq \||A|\|_{(k)}\text{ for all } t\in \R.\label{eq 2.1.18}\end{equation}
 Let $S: \R \rightarrow \mat$ be the map given by $S(t)=tU^*B$, $L:\R \rightarrow \mat$ be the map defined as $L(t)=|A|+tU^*B$ and $g:\mat \rightarrow \R_+$ be the map defined by $g(X)=\|X\|_{(k)}$.  Then we have that $g\circ L$ attains its minimum at zero. By Proposition \ref{p2.3}, we obtain that $0\in \partial (g\circ L)(0)$.  Using Proposition \ref{subdiff chain rule}, we obtain
\begin{equation}
0\in S^*\partial g(|A|).\label{eq 2.1.20}
\end{equation}

By Corollary \ref{cor 1.5.15}, this is equivalent to saying that
$$0\in \conv {\left\{\h \sum_{i=1}^{k} \langle u_i, U^*B u_i\rangle: u_1,\ldots,u_k \in \C^n, u_1,\ldots,u_k \text{ o.n.},|A| u_i=\lambda_i(|A|) u_i   \text{ for all }1\leq i\leq k \right\}}.$$
The set in the above equation is $\conv (\h \mathcal W(U^*B,|A|))$.
By Lemma \ref{lemma 2.1.7},  $\h \mathcal W(U^*B,|A|)$ is a convex set. 
So there exist $k$ orthonormal vectors $ u_1,\ldots,u_k $ such that $$|A| u_i=s_i(A) u_i $$and $$\h\sum_{i=1}^k \langle u_i, U^* B u_i\rangle=0.$$

\end{proof}

{\it Proof of Theorem \ref{thm 1}} \quad
Suppose that there exist $k$ orthonormal vectors $u_1,u_2,\ldots,u_k$ satisfying \eqref{cond1} and \eqref{cond2}. Let $\lambda \in \C$. Then similar to the argument in the proof of Theorem \ref{thm 3.2},  we get \begin{eqnarray*}
\|A+\lambda B\|_{(k)}&=& \||A|+\lambda U^*B\|_{(k)}\\
&\geq& \left|\sum_{i=1}^k \langle u_i, (|A|+ \lambda U^*B)u_i\rangle\right|\\
 &=& \left|\sum_{i=1}^k \langle u_i, |A| u_i\rangle+\lambda \sum_{i=1}^k \langle u_i, U^*B u_i\rangle\right|\\
&=& \sum_{i=1}^k s_i(A)\\
&=& \|A\|_{(k)}.
\end{eqnarray*} So $A$ is orthogonal to $B$ in $\|\cdot\|_{(k)}$.

Conversely, let $s_k(A)>0$ and $A$ is orthogonal to $B$ in $\|\cdot\|_{(k)}$. So 
$$\||A|+r e^{i\theta} U^* B\|_{(k)}\geq \|A\|_{(k)}\text{ for all }r,\theta \in \mathbb R  .$$
For $\theta\in \mathbb R$, let $B^{(\theta)}=e^{i\theta} B$. Then we get
$$\||A|+r U^* B^{(\theta)}\|_{(k)}\geq \|A\|_{(k)}\text{ for all }r \in \mathbb R  .$$
By Theorem \ref{thm 3.2}, there exist $k$ orthonormal vectors $u_1^{(\theta)},\ldots,u_k^{(\theta)}$ such that
$$|A|u_j^{(\theta)}=s_j(A) u_j^{(\theta)} \text{ for all } 1\leq j\leq k$$
and 
\begin{equation}
\h \sum_{j=1}^k \langle u_j^{(\theta)},U^*B^{(\theta)} u_j^{(\theta)}\rangle=0, \text{ that is}, \h e^{i \theta}\sum_{j=1}^k \langle u_j^{(\theta)},U^*B u_j^{(\theta)}\rangle=0. \label{complex} \end{equation}

Now by Lemma \ref{lemma 2.1.7}, the set $\mathcal W(U^*B,|A|)$
% =\left\{\sum_{j=1}^k \langle u_j, B u_j\rangle : u_1,\ldots,u_k \in \Hil ,\ u_1,\ldots,u_k\ \text{orthonormal},\ \|A\|_{(k)}=\sum_{j=1}^k \langle u_j,A u_j\rangle \right \}$$ 
is  convex in $\C$. It is also compact in $\C$.  If $0\notin \mathcal W(U^*B,|A|)$, then by the Separating Hyperplane Theorem,  there exists a $\theta_0$ such that
$$\h e^{i\theta_0}\sum_{j=1}^k \langle u_j,U^*B u_j\rangle >0\text { for all } u_1,\ldots,u_k  \text { o.n., } |A| u_j= s_j(A) u_j \text{ for } 1\leq j\leq k .$$
This is a contradiction to \eqref{complex}. 
Thus $0\in \mathcal W(U^*B,|A|)$, and so there exist $k$ orthonormal vectors $u_1,\ldots,u_k$ such that 
$$
|A|u_i=s_i(A) u_i \text{ for all } 1\leq i\leq k$$
and 
$$
\sum_{i=1}^k \langle u_i,U^*B u_i\rangle=0. $$

{\it Proof of Theorem \ref{thm 2}} \quad
 Let $S, L:\mathbb C \rightarrow \mathbb M(n)$ and $g:\mat \rightarrow \R_+$  be the maps defined as $S(\lambda)=\lambda B$, $L(\lambda)=A+\lambda B$ and $g(X)=\|X\|_{(k)}$. Then we get $\|A+\lambda B\|_{(k)}\geq \|A\|_{(k)}\text{ for all }\lambda \in \C$ if and only if  $g\circ L$ attains its minimum at 0. By Proposition \ref{p2.3} and Proposition \ref{subdiff chain rule}, a necessary and sufficient condition for this is that $0\in S^* \partial g(A)$.  Let the matrices $U, V$ be partitioned as $U=[U_1: U_2:U_3]$ and $V=[V_1:V_2:V_3]$, where $U_1,V_1\in \mathbb M(n, k-q); U_2,V_2\in \mathbb M(n, r+q); U_3,V_3\in \mathbb M(n, n-k-r)$. If $ s_k(A)>0$, then by Theorem \ref{watson ky fan},  we get that $0\in S^* \partial g(A)$ if and only if there exists a positive semidefinite matrix $T\in \mathbb M(r+q) $ with $\lambda_1(T)\leq 1$ and $ \sum_{j=1}^{r+q} \lambda_j(T)=q $  such that $\tr B^*(U_1 V_1^*+U_2 T V_2^*)=0$. Similarly, when $ s_k(A)=0$, we get that $0\in S^* \partial g(A)$ if and only if there exists $T\in \mathbb M(n-k+q,r+q) $ with $s_1(T)\leq 1$ and $\sum_{j=1}^{r+q} s_j(T)\leq q$ such that $ \tr B^*(U_1 V_1^*+[U_2:U_3] T V_2^*)=0$. A calculation shows that 
$$\tr B^*(U_1 V_1^*+U_2 T V_2^*)= \tr B_{11}^*+\tr\left( B_{22}^* T\right)$$
and $$\tr B^*(U_1 V_1^*+[U_2:U_3] T V_2^*)= \tr B_{11}^*+\tr\left(\left[ B_{22}^* : B_{32}^* \right]T\right).$$ This gives the required result. 
%We use arguments similar to the ones used in the proof of Theorem \ref{thm 1}. We first note that both the sets  $$\left\{\tr(B_{11}+T^* B_{22}): T\in \mathbb M(r+q), T\text{ positive semidefinite }, \lambda_1(T)\leq 1, \sum_{j=1}^{r+q} \lambda_j(T)=q \right\}$$ and $$\left\{\tr\left(B_{11}+T^* \left[\begin{array}{ccc}B_{22}\\ B_{32}\end{array}\right]\right): T\in \mathbb M(n-k+q,r+q), s_1(T)\leq 1, \sum_{j=1}^{r+q} s_j(T)\leq q\right\},$$

%actually these sets are $\{\langle B, G\rangle: G\in \partial g(A)\}$. \partial g(A) is compact convex and continuous image of compact set is compact.

%corresponding to the two different possibilities in the statement, are both compact convex sets in $\mathbb C$. So it is enough to show that$\|A+t B\|_{(k)}\geq \|A\|_{(k)}\text{ for all }t\in \R$ if and only if  

%the matrix \begin{equation*}B=U\left[\begin{array}{ccc} B_{11} &B_{12} & B_{13}\\ B_{21} & B_{22} & B_{23} \\ B_{31} & B_{32} & B_{33} \end{array}\right]V^*, \text{ where } B_{11}\in \mathbb M(k-t),  B_{22}\in \mathbb M(s+t),  B_{33}\in \mathbb M(n-k-s)\end{equation*} satisfies 

%$\h \tr (B_{11}+T^* B_{22})=0$ or $\h \tr \left( B_{11}+ T^* \left[\begin{array}{ccc}B_{22}\\ B_{32}\end{array}\right]\right)=0$ where $T$ satisfies the conditions given in the statement of the theorem.

\smallskip

{\it Proof of Theorem \ref{thm 3}} \quad First suppose that there exist density matrices $P_1,\ldots,P_k$ such that $\|\sum_{i=1}^k P_i\|_{\infty}\leq 1$, \begin{equation}
|A|P_i=s_i(A) P_i \quad \text{ for all } 1\leq i\leq k \label{*}
\end{equation} and $U \sum_{i=1}^k P_i \in \W^{\perp}$. Let $Q=\sum_{i=1}^k P_i$. Then $Q$ is a positive semidefinite matrix such that $\|Q\|_{\infty}\leq 1$, $\frac{1}{k}\|Q\|_{1}=\frac{1}{k} \sum_{i=1}^k \tr P_i=1$ and  \begin{equation}
\tr(W^*UQ)=0 \text{ for all }W\in \mathscr W.\label{**}
\end{equation}
 So by using \eqref{supnorm} and the fact that $\|X\|_{(k)}^*=\max\{\|X\|_{\infty}, \frac{1}{k}\|X\|_1\}$ \cite[Ex. IV.2.12]{bhatiama}., we get that for any $W\in \mathscr W$,
\begin{eqnarray*}
\|A+W\|_{(k)}&=&\||A|+U^*W\|_{(k)}\\
&\geq& \tr(|A|Q+U^*WQ)\\
&=&\tr(|A|Q) \qquad \text{ (by \eqref{**})}\\
&=&\sum_{i=1}^k \tr |A| P_i\\
&=& \|A\|_{(k)} \qquad \text{ (by \eqref{*})}.
\end{eqnarray*}
Conversely, suppose $A$ is orthogonal to $\mathscr W$ in $\|\cdot\|_{(k)}$ and $s_k(A)>0$.
Define $S:\mathscr W\rightarrow \mathbb M(n)$ as $S(W)=U^*W$. Then $S^*:\mathbb M(n)\rightarrow \mathscr W$ is given by $S^*(T)=\mathscr P_{\mathscr W}(UT)$, where $\mathscr P_{\mathscr W}$ is the orthogonal projection onto the subspace $\mathscr W$. Let $L: \mathscr W\rightarrow \mathbb M(n)$ be the map defined as $L(W)=|A|+U^*W$ and let $g:\mat \rightarrow \R_+$ be the map defined as $g(X)=\|X\|_{(k)}$. Then by Proposition \ref{p2.3} and Proposition \ref{subdiff chain rule}, we have that $\|A+W\|_{(k)}\geq \|A\|_{(k)}$ for all $W\in \mathscr W$ if and only if $0\in S^* \partial g(|A|)$. By Corollary \ref{cor 1.5.15},  there exist numbers $t_1,\ldots,t_m$ such that $0\leq t_j\leq 1$, $\sum_{j=1}^m t_j=1$ and for each $1\leq j\leq m$, there exist $k$ orthonormal vectors $u^{(j)}_{1},\ldots, u^{(j)}_k$ such that \begin{equation}
|A|u^{(j)}_i=s_i(A)u^{(j)}_i \text{ for all } 1\leq i\leq k \label{3.11}
\end{equation}  and 
\begin{equation}
S^*\left(\sum_{i=1}^k \sum_{j=1}^m t_j u^{(j)}_i u_i^{(j)*}\right)=0.\label{condition}
\end{equation} 
Let $P_i=\sum_{j=1}^m t_j u^{(j)}_i u_i^{(j)*}$. Then each $P_i$ is a density matrix. Also, by \eqref{3.11}, we get $|A|P_i=s_i(A)P_i$.
Equation \eqref{condition} says that $S^*(\sum_{i=1}^k P_i)=0$, which is equivalent to saying that $U\sum_{i=1}^k P_i\in \mathscr W^{\perp}$.
For each $1\leq j\leq m$, the matrix $\sum_{i=1}^k u^{(j)}_i u_i^{(j)*}$ is an orthogonal  projection of rank $k$ onto the linear span of $\{ u^{(j)}_i: 1\leq i\leq k\}$. In particular $\|\sum_{i=1}^k u^{(j)}_i u_i^{(j)*}\|_{\infty}\leq 1$. Thus 
\begin{eqnarray*}
\|\sum_{i=1}^k P_i\|_{\infty}&=&\| \sum_{j=1}^m t_j \sum_{i=1}^k u^{(j)}_i u_i^{(j)*}\|_{\infty}\\
&\leq & \sum_{j=1}^m t_j \|\sum_{i=1}^k u^{(j)}_i u_i^{(j)*}\|_{\infty}\\
&\leq& 1.
\end{eqnarray*}
\section{Remarks}\label{remarks}
\begin{enumerate}

\item \label{2.5}
Another necessary and sufficient condition for $A$ to be orthogonal to $B$  in $\|\cdot\|_1$ given  in \cite{lischneider} is that there exists a matrix $G\in \mat$ such that $\|G\|_{\infty}\leq 1$, $\tr(G^* A)=\|A\|_{1}$ and $\tr(G^* B)=0$. One can derive an analogous characterization for  orthogonality in $\|\cdot\|_{(k)}$ using \eqref{eq 1.4.1}. We can show that $A$ is orthogonal to $B$ in $\|\cdot\|_{(k)}$ if and only if there exists a matrix $G\in \mat$ such that $\|G\|_{\infty}\leq 1$, $\|G\|_1\leq k$, $ \tr(G^* A)=\|A\|_{(k)}$ and $ \tr(G^* B)=0$. Let $S, L, g$ be the maps as defined above in the proof of Theorem \ref{thm 2}. Then Proposition \ref{p2.3}, Proposition \ref{subdiff chain rule} and \eqref{eq 1.4.1} gives that 
% $\|A+t B\|_{(k)}\geq \|A\|_{(k)}$ if and only if 
$A$ is orthogonal to $B$ in $\|\cdot\|_{(k)}$ if and only if there exists a matrix $G\in \mat$ such that $\|G\|_{\infty}\leq 1$, $\|G\|_1\leq k$, $\h \tr(G^* A)=\|A\|_{(k)}$ and $ \tr(G^* B)=0$. We observe that if $\|G\|_{(k)}^*\leq 1$ then $\h \tr(G^* A)=\|A\|_{(k)}$ if and only if $\tr(G^* A)=\|A\|_{(k)}$. This is because if $\h \tr(G^* A)=\|A\|_{(k)}$, then $$\|A\|_{(k)}\leq |\tr(G^* A)|\leq \|G\|_{(k)}^*  \|A\|_{(k)}\leq \|A\|_{(k)}.$$ So ${\rm{Im}}\tr(G^* A)=0$ and hence $\tr(G^* A)=\h \tr(G^* A)=\|A\|_{(k)}$.
%Now if $A$ is orthogonal to $B$, then for every $\theta\in \R$, we have $\|A+r B^{(\theta)}\|\geq \|A\|$ for all $r\in \R$, where $B^{(\theta)}=e^{i \theta} B$. So there exists a matrix $G^{(\theta)}\in \mat$ such that $\|G^{(\theta)}\|_{\infty}\leq 1$, $\|G^{(\theta)}\|_1\leq k$, $\h \tr(G^{{(\theta)}^*} A)=\|A\|_{(k)}$ and $\h \tr(G^{{(\theta)}^*} B)=0$.   we obtain that there exists a matrix $G\in \mat$ such that $\|G\|_{\infty}\leq 1$, $\|G\|_1\leq k$, $\tr(G^* A)=\|A\|_{(k)}$ and $\tr(G^* B)=0$. 
Thus we obtain the required result.

\item \label{3} The characterizations for Birkhoff-James orthogonality are closely related to the recent work in norm parallelism \cite{seddik,moslehianzamani1, moslehianzamani2}. In a normed linear space, an element $x$ is said to be \emph{norm-parallel} to another element $y$ (denoted as $x||y$) if $\|x+\lambda y\|=\|x\|+\|y\|$ for some $\lambda \in \C, |\lambda|=1$. Let $A=U|A|$ be a polar decomposition of $A$ and $s_k(A)>0$. Then by Theorem 2.4 in \cite{moslehianzamani2} and 
%,we get that $A || B$ in $\|\cdot\|_{(k)}$ if and only if there exists a $\lambda \in \C$ with $|\lambda|=1$ such that $A$ is orthogonal to $\|B\|A+\lambda \|A\| B$.   we get by 
Theorem \ref{thm 1}, we get that $A || B$  in $\|\cdot\|_{(k)}$ if and only if there exists $\lambda \in \C$ with $|\lambda|=1$ and $k$ orthonormal vectors $u_1,u_2,\ldots,u_k$ such that 
$|A|\ u_i=s_i(A)u_i \text{ for all } 1\leq i\leq k$
and 
$\sum_{i=1}^k \langle u_i, U^* (\|B\|_{(k)} A+\lambda \|A\|_{(k)} B) u_i\rangle=0$.
Simplifying the expressions and using the fact that $|\lambda|=1$, we obtain that $A || B$  in $\|\cdot\|_{(k)}$ if and only if there exist $k$ orthonormal vectors $u_1,u_2,\ldots,u_k$ such that 
$|A|\ u_i=s_i(A)u_i \text{ for all } 1\leq i\leq k$ and
$|\sum_{i=1}^k \langle u_i,U^* B u_i\rangle|=\|B\|_{(k)}$. For $k=1$, this is just Corollary 2.15 of \cite{moslehianzamani2}.
\end{enumerate}

\textbf{Acknowledgement}
I would like to thank the referee for several valuable comments and suggestions.


\begin{thebibliography}{8}
\bibitem{rajic1}
L. Aramba\v{s}i\'{c}, R. Raji\'c, The Birkhoff-James orthogonality in Hilbert C$^*$-modules, {\it Linear Algebra Appl.} 437 (2012) 1913--1929.
\bibitem{rajic2}
L. Aramba\v{s}i\'{c}, R. Raji\'c, A strong version of the Birkhoff-James orthogonality in Hilbert C$^*$-modules, {\it Ann. Funct. Anal.} 5 (2014) 109--120.
\bibitem{rajic3}
L. Aramba\v{s}i\'{c}, R. Raji\'c, On three concepts of orthogonality in Hilbert C$^*$-modules, {\it Linear Multilinear Algebra} 63 (2015) 1485--1500.
\bibitem{bhatiama}
R. Bhatia, \textit{Matrix Analysis}, Springer, New York, 1997.
\bibitem{bhatiasemrl}
R. Bhatia, P. $\check{\text S}$emrl, Orthogonality of Matrices and Some Distance Problems, \textit{Linear Algebra Appl.} 287 (1999) 77-86.
\bibitem{benitez}
C. Ben\'itez, M. Fern\'andez, M.L. Soriano, { Orthogonality of matrices,} {\it Linear Algebra Appl.} 422 (2007) 155--163.
\bibitem{tbgrover}
T. Bhattacharyya, P. Grover, Characterization of Birkhoff-James orthogonality, {\it J. Math. Anal. Appl.} 407 (2013) 350--358.
\bibitem{halmos}
P.R. Halmos, \textit{A Hilbert Space Problem Book}, Narosa Publishing House, 1978.
\bibitem{hiriart}
J.B. Hiriart-Urruty, C. Lemar$\grave{\text e}$chal, \textit{Fundamentals of Convex Analysis}, Springer, 2000.
\bibitem{grover}
P. Grover, Orthogonality to matrix subspaces, and a distance formula, {\it Linear Algebra Appl.} 445 (2014) 280--288.
\bibitem{keckic}
D.J. Ke$\check{\text c}$ki$\grave{\text c}$, Gateaux derivative of B(H) norm, \textit{Proc. Amer. Math. Soc.} 133 (2005) 2061-2067.
%\bibitem{lewis} A.S. Lewis, H.S. Sendov, Nonsmooth analysis of singular values. Part II: Applications, \textit{Set-Valued Analysis} 13 (2005) 243--264.
\bibitem{lischneider}
C.K. Li, H. Schneider, Orthogonality of Matrices, \textit{Linear Algebra Appl.} 347 (2002) 115--122.
%\bibitem{liesen} J. Liesen, P. Tich$\acute{\text y}$, On best approximations of polynomials in matrices in the matrix 2-norm, {\it SIAM J. Matrix Anal. Appl.}  31 (2009) 853--863.
\bibitem{marshallolkin}
A.W. Marshall, I. Olkin, B. C. Arnold \textit{Inequalities: Theory of Majorization and Its Applications}, Springer, 2011.
%\bibitem{rockafellar}
%R.T. Rockafellar, \textit{Convex Analysis}, Princeton University Press, 1972.
\bibitem{moslehianzamani2}
M.S. Moslehian, A. Zamani, Norm-parallelism in the geometry of Hilbert C$^*$-modules, {\it Indag. Math.} 27 (2016) 266--281.
\bibitem{sainpaul}
D. Sain, K. Paul, Operator norm attainment and inner product spaces,  \textit{Linear Algebra Appl.} 439 (2013) 2448--2452.
\bibitem{sainpaulhait}
D. Sain, K. Paul, S. Hait, Operator norm attainment and Birkhoff-James orthogonality, \textit{Linear Algebra Appl.} 476 (2015) 85--97.
\bibitem{seddik}
A. Seddik, Rank one operators and norm of elementary operators, \textit{Linear Algebra Appl.} 424 (2007) 177--183.
\bibitem{singer} 
I. Singer, \textit{Best Approximation in Normed Linear Spaces by Elements of Linear Subspaces}, Springer, 1970.
\bibitem{watson}
G.A. Watson, Characterization of the Subdifferential of Some Matrix Norms, \textit{Linear Algebra Appl.} 170 (1992) 33-45.
\bibitem{watson93}
G.A. Watson, On matrix approximation problems with Ky Fan $k$ norms, \textit{Numer. Algo.} 5 (1993) 263--272.
%\bibitem{watson1} G.A. Watson, \textit{Approximation Theory and Numerical Methods}, John Wiley \& Sons, 1980.
\bibitem{zalinescu}
C. Z\u{a}linescu, \textit{Convex Analysis in General Vector Spaces}, World Scientific, Singapore, 2002.
\bibitem{moslehianzamani1}
A. Zamani, M.S. Moslehian, Exact and approximate operator parallelism, {\it Canad. Math. Bull.} 58 (2015)
207--224.
%\bibitem{zietak1} K. Zietak, Properties of linear approximations of matrices in the spectral norm, {\it Linear Algebra Appl.} 183 (1993) 41--60. 
%\bibitem{zietak2} K. Zietak, On approximation problems with zero-trace matrices, {\it Linear Algebra Appl.} 247 (1996) 169--183.
\end{thebibliography}
\end{document}